\documentclass[12pt]{amsart}
\usepackage{amssymb,amsmath}
\usepackage{amsfonts}
\usepackage{epsf}
\usepackage{graphicx}
\newtheorem{theorem}{Theorem}[section]
\newtheorem{lemma}[theorem]{Lemma}

\newtheorem{remark}[theorem]{Remark}

\theoremstyle{definition}

\numberwithin{equation}{section}

\newcommand{\bc}{\begin{center}}
\newcommand{\ec}{\end{center}}
\newcommand{\be}{\begin{eqnarray}}
\newcommand{\ee}{\end{eqnarray}}

\newcommand{\ben}{\begin{eqnarray*}}
\newcommand{\een}{\end{eqnarray*}}

\def\hK{\hat{K}}

\def\cT{\mathcal{T}}

\def\ad#1{\begin{aligned}#1\end{aligned}}  
\def\a#1{\begin{align*}#1\end{align*}} \def\an#1{\begin{align}#1\end{align}}
\def\e#1{\begin{equation}#1\end{equation}} 
 
  \numberwithin{equation}{section}
\def\boxit#1{\vbox{\hrule height1pt \hbox{\vrule width1pt\kern1pt
     #1\kern1pt\vrule width1pt}\hrule height1pt }}
 \def\lab#1{\boxit{\small #1}\label{#1}} \def\mref#1{\boxit{\small #1}\ref{#1}}
 \def\meqref#1{\boxit{\small #1}\eqref{#1}}

  \def\lab#1{\label{#1}} \def\mref#1{\ref{#1}} \def\meqref#1{\eqref{#1}}
\long\def\comments#1{}
\title[]
{Nonconforming finite element methods on quadrilateral meshes}
\author[J.~Hu]{Jun Hu}
\address{LMAM and  School of Mathematical Sciences,
  Peking University,  Beijing 100871, P. R. China. hujun@math.pku.edu.cn}
\author {Shangyou Zhang}
\address{Department of Mathematical Sciences, University of Delaware,
    Newark, DE 19716, USA.  szhang@udel.edu }

\thanks{The first author was supported by  the NSFC Project 11271035, and  in part by the NSFC Key Project 11031006.}

\begin{document}
\begin{abstract}
It is well-known that  it is  comparatively difficult to design nonconforming finite elements on quadrilateral meshes by using
Gauss-Legendre points  on each  edge of triangulations. One reason lies in that these degrees of freedom
associated to these  Gauss-Legendre points are not all linearly independent for usual expected polynomial spaces,
which explains why only  several lower order nonconforming quadrilateral finite elements can be found in literature.
The present paper  proposes two  families of nonconforming finite elements of any odd order and one family of
nonconforming finite elements of any even order on quadrilateral meshes. Degrees of freedom are given for these elements, which are
proved to be well-defined for their corresponding shape function spaces in a unifying way.  These elements generalize
three lower order nonconforming finite elements on quadrilaterals to any order. In addition, these nonconforming
 finite element spaces are shown to be full spaces which is somehow not discussed for nonconforming finite elements
 in literature before.
\end{abstract}
\maketitle

\section{Introduction}
Because of their  flexibility and stability when compared with conforming finite element methods, nonconforming
finite element methods have become very important and effective discretization methods for numerically
solving, among others,  high order elliptic problems, Stokes-like problems and  Reissner-Mindlin plate bending problems.

Quadrilateral meshes are very important in scientific and  engineering computing. Indeed,
many popular softwares for computations of the fluid mechanics in two dimensions are defined on quadrilateral meshes.
However, most of nonconforming finite element methods for second order problems are  defined on  triangles \cite{Arnold,BaranSoyan,Ciarlet(1978),CrouzeixRaviart(1973),FortinSoulie(1983)}  while there are only
a few nonconforming finite element methods on quadrilaterals.

Compared with  nonconforming triangular finite elements, it is more difficult to construct nonconforming quadrilateral finite elements.
In fact,  a sufficient condition for convergence of consistency error terms is to require nonconforming functions
to  be  continuous at Gauss-Legendre points on interior edges of the triangulation used \cite{Shi(1987),ShiWang(2010),Stummel(1979),Wang(2002)}.
Hence, for $m$ order nonconforming finite elements,  there are $4m$ Gauss-Legendre points on the boundary of  each element.
However, these degrees of freedom on  $4m$ Gauss-Legendre points are not all linearly independent for the space of polynomials
whose  restrictions on four edges are  polynomials of degree $\leq m$,
  see \eqref{Rel} for details.
 This means that at least one relation holds, which motivates
the use of higher order monomials of one variable in shape function spaces. In such a spirit, a class of
nonconforming quadrilateral elements is proposed in literature, which includes, the Han element \cite{Han(1984)},
the nonconforming  rotated $Q_1$ element due to Rannacher and Turek \cite{RannacherTurek(1992)}, the Douglas-Santos-Sheen-Ye (DSSY) element \cite{DouSanSheYe99}, the  enriched  nonconforming  rotated $Q_1$ element due to Lin, Tobiska and Zhou \cite{Ltz04}. All of these nonconforming quadrilateral  finite element methods are of first order and  are stable for the Stokes problem with pure Dirichlet boundary conditions for the velocity. However, since the usual Korn inequality does not hold for  them, they can not be applied to the Stokes problem with
mixed  boundary conditions for the velocity.  Moreover, these nonconforming elements are somehow  not extended to higher order
 nonconforming elements in literature so far .

In \cite{ParkSheen(2003)}, a nonconforming linear element is introduced on quadrilateral meshes,
which is motivated by a key observation that any linear function on a quadrilateral
can be uniquely determined at any three of the four midpoints of edges, which leads to a set of  nodal basis
functions whose values at four edges satisfies the aforementioned relation.  A similar  element (but different on general quadrilateral meshes)
is designed in \cite{HuShi(2005)}, which is based on an observation that a frame of the linear function space on the  reference element
can be mapped and glued to form a basis of the interpolated space of the conforming bilinear element space by
the canonical interpolation operator of the nonconforming  rotated $Q_1$ element \cite{RannacherTurek(1992)}. This explains
why the element therein is named as ``constrained quadrilateral nonconforming rotated $Q_1$ element" by its authors.
Very recently, the ideas of \cite{HuShi(2005),ParkSheen(2003)}  are extended to design
a nonconforming cubic element on quadrilaterals in \cite{MengLuoSheen}. We refer interested readers to \cite{KimMengNamParkSheen2013,LeeSheen2006}
for the nonconforming quadratic element on quadrilaterals, which can be regarded as the quadrilateral counterpart of the triangular Fortin-Soulie element \cite{FortinSoulie(1983)}.

The purpose paper is to propose nonconforming finite elements of any order on quadrilaterals in a unifying way, which is somehow missed
 in literature.  Compared with nonconforming finite elements on triangles \cite{Arnold,BaranSoyan,Ciarlet(1978),CrouzeixRaviart(1973),FortinSoulie(1983)}, there are at least two more difficulties  on quadrilaterals:
 (1) what shape function spaces should be used for nonconforming finite elements of any order;
 (2) how to prove unisolvency of these degrees of freedom which consist of values of polynomials at aforementioned $4m$ Gauss-Legendre
  points  and  other degrees of freedom in the interior of  elements.  It should be stressed that nonconforming finite elements
  on quadrilaterals in literature are defined and analyzed one by one, see  \cite{DouSanSheYe99,Han(1984),HuShi(2005),KimMengNamParkSheen2013,LeeSheen2006,Ltz04,MengLuoSheen,ParkSheen(2003),RannacherTurek(1992)}.

 In order to overcome the first difficulty, we propose to use the following two families of  shape function spaces
\begin{equation}
R_m(\hat{K}):=P_m(\hat{K})+\text{span}\{\hat{x}^m\hat{y}-\hat{x} \hat{y}^m\},
\end{equation}
\begin{equation}
ER_m(\hat{K}):=P_m(\hat{K})+\text{span}\{\hat{x}^m\hat{y}-\hat{x} \hat{y}^m, \hat{x}^{m+1}- \hat{y}^{m+1} \},
\end{equation}
for odd integer $m\geq 1$. Here and  throughout this paper, $P_m(M)$ denotes the space of polynomials of degree $\leq m$ over the domain $M$;
$Q_m(M)$ denotes the space of polynomials of degree $\leq m$ in each variable. For even $m$, we propose to use
 the same shape function  spaces as  the serendipity elements of \cite{Arnold-Awanou2011,Ciarlet(1978)},
 namely,
\begin{equation}
R_m^{+}(\hat{K}):=P_m(\hat{K})+\text{span}\{\hat{x}^m\hat{y}, \hat{x}\hat{y}^m\}.
\end{equation}
These three families of elements generalize the $P_1$ nonconforming element of \cite{HuShi(2005),ParkSheen(2003)} and the nonconforming
 cubic element of \cite{MengLuoSheen},  the nonconforming rotated $Q_1$ element of \cite{RannacherTurek(1992)}, and the nonconforming quadratic element of \cite{KimMengNamParkSheen2013,LeeSheen2006}  to any order, respectively.

 Degrees of freedom are given for these shape function spaces, which are proved to be well defined based on some observation concerning
  Legendre polynomials on four edges. In addition, these finite element spaces are proved to be full spaces.

The rest of the paper is organized as follows.  Sections 2, 3 and 4
   present three families of shape function spaces and their corresponding
degrees of freedom, which are proved to be well-defined.   Section 5 defines  nonconforming finite element spaces which are proved to
 be full spaces, and shows  approximations of these spaces.
 Section 6  analyzes  consistency errors,
    which is followed by numerical examples in the final section.

\section{The first family of Shape function spaces}
 Since Gauss-Legendre points on each edge will be used to define degrees of freedom and consequently continuity  for new elements under consideration,  given an  integer $k\geq 0$,  let  $g_i$,  $i=-k, \cdots,-1, 0, 1,\cdots, k$, denote zeros of
  Legendre polynomials of degrees $2k+1$ on the interval $[-1,1]$.
By skew  symmetry, $g_i=-g_{-i}$, $i=1, \cdots, k$, and $g_0=0$.

Before we present  shape function spaces on $\hK:=[-1,1]^2$ and corresponding degrees of freedom,
we  investigate a special  relation of values  at $8k+4$  Gauss-Legendre points
$G_{r,i}=(1, g_i)$, $G_{l,i}=(-1,g_i)$, $G_{t,i}=(g_i,1)$, $G_{b,i}=(g_i, -1)$, $i=-k, \cdots, k$ on four edges of $\hK$,
of polynomials over $\hK$ whose restrictions on four edges are polynomials of degree $\leq 2k+1$.

\begin{lemma} Let $\hat{v}$ be a polynomial over $\hK$ such that its restrictions on four edges are polynomials of degree $\leq 2k+1$.
Then it holds that
\begin{equation}\lab{Rel}
 \begin{split}
 \sum\limits_{i=-k}^{k}\gamma_i\bigg(\hat{v}(1,g_i)+\hat{v}(-1,g_i)\bigg)= \sum\limits_{i=-k}^{k}\gamma_i\bigg(\hat{v}(g_i,1)+\hat{v}(g_i,-1)\bigg),
\end{split}
 \end{equation}
 where
 \begin{equation}
 \gamma_i=\prod\limits_{i\not=j=-(k+1)}^{k}\frac{1-g_j}{g_i-g_j}+\prod\limits_{i\not=j=-k}^{k+1}\frac{-1-g_j}{g_i-g_j}
 =\frac{2}{g_i^2}\prod\limits_{|i|\not=j=1}^{k}\frac{1-g_j^2}{g_i^2-g_j^2},
 \end{equation}
 for $i=-k, \cdots, -1, 1, \cdots, k$, and
 \begin{equation}
 \gamma_0=4\prod\limits_{j=1}^{k}\frac{g_j^2-1}{g_j^2}.
 \end{equation}
\end{lemma}
\begin{remark} For k=0, the relation \eqref{Rel} is the constraint used
   in \cite{HuShi(2005),ParkSheen(2003)};
  for k=1, the relation \eqref{Rel}  recovers that of \cite{MengLuoSheen}.
 \end{remark}
\begin{proof} The main idea is: Since the restrictions of $\hat{v}$ on four edges are polynomials of degree $\leq 2k+1$, they can be exactly
represented by  $2k+2$ Lagrange interpolation basis functions of degree $\leq 2k+1$ in one dimension, which
gives two representations of values of $\hat{v}$ at each corner of $\hK$,  then the desired result follows.
  In order to accomplish this,  first add the point $g_{-k-1}=-1$ to these $2k+1$ Gauss-Legendre points,
and  define a first set of  Lagrange interpolation basis functions:
 \begin{equation*}
 \mathcal{L}_i^0(\hat{x})=\prod\limits_{i\not=j=-(k+1)}^{k}\frac{\hat{x}-g_j}{g_i-g_j}  \text{ for }\hat{x}\in [-1,1], i=-(k+1), -k, \cdots, k.
 \end{equation*}
  Then, add the point $g_{k+1}=1$ to these $2k+1$ Gauss-Legendre points, and  define another set of  Lagrange interpolation basis functions:
 \begin{equation*}
 \mathcal{L}_i^1(\hat{x})=\prod\limits_{i\not=j=-k}^{k+1}\frac{\hat{x}-g_j}{g_i-g_j}  \text{ for }\hat{x}\in [-1,1], i=-k, \cdots,  k+1.
 \end{equation*}
 Since $\text{span}\{\mathcal{L}_{-(k+1)}^0(\hat{x}), \cdots, \mathcal{L}_{k}^0(\hat{x})\}=\text{span}\{\mathcal{L}_{-k}^1(\hat{x}), \cdots, \mathcal{L}_{k+1}^1(\hat{x})\}=P_{2k+1}([-1,1])$, and  the restriction of  $\hat{v}$  on each edge of $\hK$ is a polynomial of degree $\leq 2k+1$,
  these restrictions can be expressed as
   \begin{equation}\label{2.4}
 \ad{
  \hat{v}(\pm 1,\hat{y})
        &=\sum\limits_{i=-(k+1)}^{k} \mathcal{L}_i^0(\hat{y})\hat{v}(\pm 1,g_i),
	& \hat{v}(\hat{x}, \pm 1)&=
 \sum\limits_{i=-(k+1)}^{k} \mathcal{L}_i^0(\hat{x})\hat{v}(g_i, \pm 1), \\
   \hat{v}(\pm 1,\hat{y})
    &=\sum\limits_{i=-k}^{k+1} \mathcal{L}_i^1(\hat{y})\hat{v}(\pm 1,g_i),
&  \hat{v}(\hat{x}, \pm 1) &=
 \sum\limits_{i=-k}^{k+1} \mathcal{L}_i^1(\hat{x})\hat{v}(g_i, \pm 1).
  }
 \end{equation}
  Let $\alpha_i= \mathcal{L}_i^0(1)$, $i=-(k+1), \cdots, k$,
   and $\beta_i=\mathcal{L}_i^1(-1)$, $i=-k, \cdots, k+1$.
  The skew symmetry of  Gauss-Legendre points,  and
    definitions of $\alpha_{-(k+1)}$ and $\beta_{k+1}$,    yield
   \begin{equation}\label{2.5}
   \alpha_{-(k+1)}=\prod\limits_{ j=-k}^{k}\frac{1-g_j}{-1-g_j}=
   \prod\limits_{ j=-k}^{k}\frac{-1-g_j}{1-g_j}=\beta_{k+1}.
   \end{equation}
   Since  values $\hat{v}(1,1)$ (resp.  $\hat{v}(-1,1)$,  $\hat{v}(1,-1)$ and  $\hat{v}(-1,-1)$) of two representations are identical,
   it follows from \eqref{2.4} that
  \begin{equation*}
  \begin{split}
   &\sum\limits_{i=-(k+1)}^k\alpha_i\hat{v}(1,g_i)
   =\sum\limits_{i=-(k+1)}^k\alpha_i\hat{v}(g_i,1),
   \quad\sum\limits_{i=-k}^{k+1}\beta_i\hat{v}(1,g_i)
       =\sum\limits_{i=-(k+1)}^k\alpha_i\hat{v}(g_i,-1),\\
    &\sum\limits_{i=-(k+1)}^k\alpha_i\hat{v}(-1, g_i)
    =\sum\limits_{i=-k}^{k+1}\beta_i\hat{v}(g_i,1),
  \quad \sum\limits_{i=k}^{k+1}\beta_i\hat{v}(-1,g_i)=\sum\limits_{i=-k}^{k+1}\beta_i\hat{v}(g_i,-1).
  \end{split}
  \end{equation*}
  By \eqref{2.5}, this gives
  \begin{equation*}
   \begin{split}
 \sum\limits_{i=-k}^{k}(\alpha_i+\beta_i)\big(\hat{v}(1,g_i)+\hat{v}(-1,g_i)\big)
= \sum\limits_{i=-k}^{k}(\alpha_i+\beta_i)\big(\hat{v}(g_i,1)+\hat{v}(g_i,-1)\big).
  \end{split}
   \end{equation*}
  Then a direct calculation completes the proof.
 \end{proof}

By canceling a common factor,  \meqref{Rel} can be simplified slightly that
\a{  \sum\limits_{i=-k}^{k}
   \frac{ \gamma_i' }{\prod_{|i|\ne j=1}^k (g_j^2-g_i^2)}
      \big(\hat{v}(1,g_i)+\hat{v}(-1,g_i)
         -\hat{v}(g_i,1)-\hat{v}(g_i,-1)\big)&=0, }
where \a{ \gamma_i'=\begin{cases} 2 & \hbox{if}\ i=0, \\
                    1/g_i^2  & \hbox{if} \ i\ne 0.\end{cases} }

Given an odd integer  $m=2k+1\geq 0$, recall  shape function spaces:
\begin{equation}
R_m(\hat{K}):=P_m(\hat{K})+\text{span}\{\hat{x}^m\hat{y}-\hat{x} \hat{y}^m\}
     \text{ for any } (\hat{x},\hat{y})\in \hat{K}:=[-1,1]^2.
\end{equation}
Note that for $k=0$, $R_m(\hK)$ is the shape function space of
    \cite{HuShi(2005)}, i.e.,
\a{  R_1 & = \text{span}\{1, \hat{x}, \hat{y} \}. }
For $k=1$, $R_m(\hat{K})$  is the shape function space of \cite{MengLuoSheen}.

 To define  degrees of freedom for this space, let
$G$ denote the set of Gauss-Legendre points on four edges of $\hK$, namely,
\begin{equation} \lab{G}
G:=\{(1, g_i),(-1,g_i), (g_i,1), (g_i, -1), i=-k, \cdots, k\}.
\end{equation}
 To define other  degrees of freedom in the interior of $\hK$,  let
\e{\lab{I}
  I:=\{(\hat{x}_\ell, \hat{y}_\ell), \ell=1, \cdots, (2k-1)(k-1)\}
} be a set of interior points of $\hat{K}$,  the standard Lagrange points
   inside the reference element  $\hat{K}$,
   so that any polynomial $\hat{q}(\hat{x},\hat{y})\in P_{2k-3}(\hat{K})$
    can be uniquely defined by its values at points in $I$.

\begin{theorem}\label{Theorem2.1} If $\hat{v}(\hat{x},\hat{y})\in R_m(\hK)$
   vanishes at all points in $G \cup I$, cf. \meqref{G} and \meqref{I},
   then $\hat{v}(\hat{x},\hat{y})\equiv 0$.
\end{theorem}
\begin{proof} The function $\hat{v}(\hat{x},\hat{y})\in R_m(\hK)$
    can be expressed as
\an{\lab{express}
   \hat{v}(\hat{x},\hat{y}) =a_0(\hat{x}^{2k+1}\hat{y}
   -\hat{x}\hat{y}^{2k+1})+ \sum_{i+j\le 2k+1} c_{i,j}\hat{x}^i \hat{y}^j. }
As $\hat{v}(1,\hat{y})=0$ at $2k+1$ Gauss-Legendre points, $\hat{v}(1,\hat{y})$ is a multiple
of the $(2k+1)$-st Legendre polynomial,
   $\hat{v}(1,\hat{y})=c_0 L_{2k+1}(\hat{y})$ for some constant $c_0$.  We will show that the constant $c_0$
is zero.  In fact, by the continuity of $\hat{v}(\hat{x},\hat{y})$ at four vertexes of the square,
\an{\lab{ex1}
   \hat{v}(-1,\hat{y})&=-c_0L_{2k+1}(\hat{y})= -\hat{v}( 1,\hat{y})
      \text{ for }\hat{y}\in [-1,1], \\
   \lab{ex2}
    \hat{v}(\hat{x},-1)&=-c_0L_{2k+1}(\hat{x})=-\hat{v}(1,\hat{x})
     \text{ for }\hat{x}\in [-1,1]. }
This indicates that
\a{
   & \left\{
   \ad{
  & c_{0,0}+c_{2,0}+c_{4,0}+c_{6,0}+\cdots+c_{2k,0}&&=0, &&(\text{ for } j=0)\\
  & c_{0,1}+c_{2,1}+c_{4,1}+c_{6,1}+\cdots+c_{2k,1}&&=0, && (\text{ for } j=1)\\
  & c_{0,2}+c_{4,2}+c_{6,2}+c_{8,2}+\cdots+c_{2k-2,2}&&=0, &&(\text{ for } j=2)\\
  & \vdots\\
  & c_{0,2k-2}+c_{2, 2k-2}&&=0, &&(\text{ for }j=2k-2)\\
  &c_{0,2k-1}+c_{2, 2k-1}&&=0, &&(\text{ for }j=2k-1)\\
 &c_{0,2k}&&=0, &&(\text{ for }j=2k)\\
  & c_{0,2k+1}&&=0. &&(\text{ for }j=2k+1)&&
    } \right.\\
   & \left\{
   \ad{
  & c_{0,0}+c_{0,2}+c_{0,4}+c_{0,6}+\cdots+c_{0, 2k}&&=0, && (\text{ for } i=0)\\
  & c_{1,0}+c_{1,2}+c_{1,4}+c_{1,6}+\cdots+c_{1,2k}&&=0, &&
          (\text{ for } i=1)\\
  & c_{2,0}+c_{2,4}+c_{2,6}+c_{2,8}+\cdots+c_{2,2k-2}&&=0, &&
              (\text{ for } i=2)\\
  & \vdots\\
  & c_{2k-2, 2}+c_{2k-2, 2}&&=0, &&(\text{ for }i=2k-2)\\
  &c_{2k-1, 0}+c_{2k-1, 2}&&=0, &&(\text{ for }i=2k-1)\\
 &c_{2k,0}&&=0, &&(\text{ for }i=2k)\\
  & c_{2k+1,0}&&=0. &&(\text{ for }i=2k+1)&&
  }  \right.
  }
   Since $c_{0,2k+1}=0$, by comparing the
      coefficients of $\hat y^{2k+1}$ in \meqref{ex1}, $a_0$ is a multiple of
     the leading coefficient of the Legendre polynomial $L_{2k+1}$:
  \a{   c_0 \frac{(4k+2)!}{2^{2k+1}(2k+1)! } = -a_0. }
  Again,  comparing the coefficients of $\hat x^{2k+1}$ in \meqref{ex2},
   \a{   c_0 \frac{(4k+2)!}{2^{2k+1}(2k+1)! } = a_0. }
 Hence $c_0=0$. Thus  $\hat{v}(\hat{x},\hat{y})$ vanishes on the whole boundary
    $\partial \hat{K}$ and can be expressed as
$$
\hat{v}(\hat{x},\hat{y})=\hat{b}(\hat{x},\hat{y})\hat{q}(\hat{x},\hat{y}) \text{ with } \hat{q}(\hat{x},\hat{y})\in P_{2k-3}(\hat{K}),
$$
where the bubble function $\hat{b}(\hat{x}, \hat{y})=(1-\hat{x}^2)(1-\hat{y}^2)$. Finally, since $\hat{v}(\hat{x},\hat{y})$ vanishes at points in $I$, $\hat{q}(\hat{x},\hat{y})\equiv 0$. This  completes the proof.
\end{proof}

The above theorem implies that any $\hat{v}(\hat{x},\hat{y})\in R_m(\hK)$ can be uniquely  determined by
its values at  points in $G\cup I$. Though the number of points in $G\cup I$ is $(2k+3)(k+1)+2$,
which is $1$ greater than the dimension $(2k+3)(k+1)+1$ of $R_m(\hat{K})$,
 the relation \eqref{Rel}  implies that the number of linearly independent functionals defined
 for the space $R_m(\hK)$ is equal to the dimension of the shape function space.
  This  motivates the following degrees of freedom for the shape function space $R_m(\hK)$:
\begin{itemize}
\item values at points in $G$ which satisfy the relation \eqref{Rel};
\item values  at points in $I$.
\end{itemize}

\begin{remark}  We can take the following shape function spaces
\begin{equation}
\tilde{R}_{m}(\hK):=P_{m}(\hK)+\text{span}\{\hat{x}^m\hat{y}\}, \text{ and }
\bar{R}_{m}(\hK):=P_{m}(\hK)+\text{span}\{\hat{x}\hat{y}^m\}.
\end{equation}
\end{remark}

\section{The second family of shape function spaces}
As we see in the previous section, functionals defined by values at points in
$G \cup I$ are not all linearly independent for shape function spaces $R_m(\hK)$
defined in the previous section,  there is a relation \eqref{Rel}
   for all functions in $R_m(\hK)$.
 To make these functionals linearly independent for some shape function spaces, we propose to use higher order
  monomials of one variable, say $\hat{x}^{2k+2}$ and $\hat{y}^{2k+2}$, which motivates to
    enrich $R_m(\hK)$ by $\text{span}\{\hat{x}^{2k+2}-\hat{y}^{2k+2}\}$.  This leads to the following shape function spaces:
\begin{equation}
ER_m(\hK):=P_{m}(\hK)+\text{span}\{\hat{x}^{2k+1}\hat{y}-\hat{x}\hat{y}^{2k+1}, \hat{x}^{2k+2}-\hat{y}^{2k+2}\}.
\end{equation}
\begin{remark}
For $k=0$, the space $ER_m(\hK)$ is the shape function space of the nonconforming rotated $Q_1$ element from \cite{RannacherTurek(1992)}:
\a{ ER_1(\hK) = \text{span}\{1, \hat{x}, \hat{y},  \hat{x}^{2}-\hat{y}^{2}\}. }
\end{remark}

For the space $ER_m(\hK)$, the  degrees of freedom are:
\begin{itemize}
\item values at  points in $G$, defined in \meqref{G};
\item values  at points in $I$, defined in \meqref{I}.
\end{itemize}
\begin{theorem}\label{Theorem3.1} If $\hat{v}(\hat{x},\hat{y})\in ER_m(\hK)$ vanishes at  points in $G\cup I$,
 then $\hat{v}(\hat{x},\hat{y})\equiv 0$.
\end{theorem}
\begin{proof} By the definition of  $ER_m(\hK)$, $\hat{v}(\hat{x}, \hat{y})$ can be expressed as
\begin{equation}\label{3.2}
\hat{v}(\hat{x}, \hat{y})=\hat{v}_1(\hat{x}, \hat{y})+c_1\hat{x}^{2k+1}+c_2\hat{y}^{2k+1}+c_3(\hat{x}^{2k+1}\hat{y}-\hat{x}\hat{y}^{2k+1})+c_4(\hat{x}^{2k+2}-\hat{y}^{2k+2}),
\end{equation}
where $\hat{v}_1(\hat{x}, \hat{y})\in P_{m}(\hK)\backslash\text{span}\{\hat{x}^{2k+1}, \hat{y}^{2k+1}\}$, and $c_i$, $i=1,\cdots, 4$,  are four interpolation parameters. This implies  that the restrictions of  $\hat{v}$ on four edges of $\hK$ are polynomials of degree $\leq 2k+2$.
 Since $\hat{v}(\hat{x}, \hat{y})$ vanishes at  $m=2k+1$ Gauss-Legendre  points on four edges of $\hat{K}$, these restrictions on four edges
  can be written as
 \begin{equation}\label{3.3}
 \begin{split}
 &\hat{v}(\hat{x}, 1)=L_{2k+1}(\hat{x})(a_1\hat{x}+b_1),
    \quad\hat{v}(\hat{x}, -1)=L_{2k+1}(\hat{x})(a_2\hat{x}+b_2),\\
 & \hat{v}(1, \hat{y})=L_{2k+1}(\hat{y})(a_3\hat{y}+b_3),
   \quad\hat{v}(-1, \hat{y})=L_{2k+1}(\hat{y})(a_4\hat{y}+b_4),
 \end{split}
 \end{equation}
 where $a_i$, $b_i$, $i=1, \cdots, 4$, are some constants which will be shown to be zero next.
To this end, let $c={(4k+2)!}/(2^{2k+1}(2k+1)!)$
    be the coefficient of the monomial $\hat{x}^{2k+1}$
   of $L_{2k+1}(\hat{x})$.
A comparison of coefficients from \eqref{3.2} and \eqref{3.3} for monomials
  $\hat{x}^{2k+2}$, $\hat{y}^{2k+2}$, $\hat{x}^{2k+1}$, and $\hat{y}^{2k+1}$ leads to
\begin{equation}\label{3.4}
   a_1=a_2=-a_3=-a_4=c_4/c,
\end{equation}
and
\begin{equation}\label{3.5}
\begin{split}
c_1+c_3=c\,b_1,\quad c_1-c_3=c\,b_2,\\
c_2-c_3=c\,b_3, \quad c_2+c_3=c\,b_4.
\end{split}
\end{equation}
 Note that $\hat{v}(\hat{x},1)$ and $\hat{v}(1,\hat{y})$ are equal at corner
$(1,1)$,  $\hat{v}(\hat{x},1)$ and $\hat{v}(-1,\hat{y})$ are equal at corner $(-1,1)$,
 $\hat{v}(\hat{x},-1)$ and $\hat{v}(1, \hat{y})$ are equal at corner $(1,-1)$, $\hat{v}(\hat{x}, -1)$ and  $\hat{v}(-1, \hat{y})$
are equal at corner $(-1,-1)$. Since $L_{2k+1}(1)=-L_{2k+1}(-1)$, these observations give
\begin{equation}\label{3.6}
\begin{split}
a_1+b_1=a_3+b_3, \quad a_1-b_1=a_4+b_4,\\
a_2+b_2=a_3-b_3, \quad a_2-b_2=a_4-b_4.
\end{split}
\end{equation}
It follows from \eqref{3.4}, \eqref{3.5} and \eqref{3.6} that
\begin{equation}
a_1=a_2=a_3=a_4=b_1=b_2=b_3=b_4=c_1=c_2=c_3=0.
\end{equation}
This implies that $c_4=0$ and $\hat{v}(\hat{x},\hat{y})$ vanishes on the boundary of $\hK$. Hence,
$$
\hat{v}(\hat{x},\hat{y})=\hat{b}(\hat{x},\hat{y})\hat{q}(\hat{x},
   \hat{y}) \text{ with } \hat{q}(\hat{x},\hat{y})\in P_{2k-3}(\hat{K}),
$$
where the bubble function $\hat{b}(\hat{x}, \hat{y})=(1-\hat{x}^2)(1-\hat{y}^2)$. Finally,
 since $\hat{v}(\hat{x},\hat{y})$ vanishes at points in $I$, $\hat{q}(\hat{x},\hat{y})\equiv 0$. This  completes the proof.
\end{proof}


For the space $ER_m(\hK)$, we define another set of  degrees of freedom as follows
\begin{itemize}
\item moments of order $\leq 2k$ on each edge of $\hK$;
\item values  at points in $I$.
\end{itemize}
We note that these two sets of degrees of freedom are not equivalent as $(2k+2)$ polynomials
    are involved in  $ER_m(\hK)$.

\begin{theorem}\label{Theorem3.5} For any $\hat{v}(\hat{x},\hat{y})\in ER_m(\hK)$, suppose its moments of order $\leq 2k$ on each edge of $\hK$
and   values  at points in $I$ vanish.  Then $\hat{v}(\hat{x},\hat{y})\equiv 0$.
\end{theorem}
\begin{proof} By the definition of  $ER_m(\hK)$, $\hat{v}(\hat{x}, \hat{y})$ can be expressed as
\begin{equation*}
\hat{v}(\hat{x}, \hat{y})=\hat{v}_1(\hat{x}, \hat{y})+c_1\hat{x}^{2k+1}+c_2\hat{y}^{2k+1}+c_3(\hat{x}^{2k+1}\hat{y}-\hat{x}\hat{y}^{2k+1})+c_4(\hat{x}^{2k+2}-\hat{y}^{2k+2}),
\end{equation*}
where $\hat{v}_1(\hat{x}, \hat{y})\in P_{m}(\hK)\backslash\text{span}\{\hat{x}^{2k+1}, \hat{y}^{2k+1}\}$, and $c_i$, $i=1,\cdots, 4$,  are four interpolation parameters. Consider the restriction on edge $\hat{e}_2$ of $\hat{v}$, denoted by $\hat{w}$, which is a polynomial of degree
 $\leq 2k+2$. The function $\hat{w}$ can be decomposed as
 \begin{equation*}
 \hat{w}=\hat{w}_1+\hat{w}_2,
 \end{equation*}
 where
 \begin{equation*}
   \hat{w}_1=\sum\limits_{i=0}^{k+1}d_{2i}\hat{x}^{2i}, \text{ and }\hat{w}_2=\sum\limits_{i=0}^{k}d_{2i+1}\hat{x}^{2i+1},
 \end{equation*}
 where $d_i$, $i=0, 1, \cdots, 2(k+1)$, are interpolation constants for $\hat{w}$.  From degrees of freedom it follows
\begin{equation*}
 \int_{-1}^1 \hat{w}_1 \hat{x}^{2i} d \hat{x}=0, i=0, \cdots, k, \text{ and }\int_{-1}^1 \hat{w}_2 \hat{x}^{2i+1} d \hat{x}=0, i=0, \cdots, k-1.
\end{equation*}
Since $\hat{w}_1$ (resp. $\hat{w}_2$) is an even (odd) function on $[-1,1]$,  it holds that
\begin{equation*}
\int_{-1}^1 \hat{w}_1 \hat{x}^{2i+1} d \hat{x}=0, i=0, \cdots, k, \text{ and }\int_{-1}^1 \hat{w}_2 \hat{x}^{2i} d \hat{x}=0, i=0, \cdots, k.
\end{equation*}
Note that the degree of polynomial $\hat{w}_1$ (resp. $\hat{w}_2$) is not more than $2k+2$ (resp. $2k+1$).
Therefore both $\hat{w}_1$ and $\hat{w}_2$ are Legendre polynomials (up to multiplication constants), namely,
\begin{equation*}
\hat{w}_1=a_2L_{2k+2}(\hat{x}), \text{ and }\hat{w}_2=b_2L_{2k+1}(\hat{x}),
\end{equation*}
for two constants $a_2$ and $b_2$.  Similar arguments apply to the restrictions on the other three edges of $\hK$, which leads to
\begin{equation*}
\begin{split}
\hat{v}|_{\hat{e}_1}=a_1L_{2k+2}(\hat{y})+b_1L_{2k+1}(\hat{y}),\\
\hat{v}|_{\hat{e}_2}=a_2L_{2k+2}(\hat{x})+b_2L_{2k+1}(\hat{x}),\\
\hat{v}|_{\hat{e}_3}=a_3L_{2k+2}(\hat{y})+b_3L_{2k+1}(\hat{y}),\\
\hat{v}|_{\hat{e}_4}=a_4L_{2k+2}(\hat{x})+b_4L_{2k+1}(\hat{x}).
\end{split}
\end{equation*}
Since the coefficient before monomial $\hat{x}^{2k+2}$ is opposite to that before  monomial $\hat{y}^{2k+2}$ for $\hat{v}$,
 this  gives
\begin{equation}\label{3.14}
a_1=a_3=-a_2=-a_4.
\end{equation}
 Let $c \not=0$ be the coefficient of monomial $\hat{x}^{2k+1}$ of $L_{2k+1}(\hat{x})$.
A comparison of coefficients  for monomials  $\hat{x}^{2k+1}$, and $\hat{y}^{2k+1}$ in $\hat{v}|_{\hat{e}_i}$, $i=1, \cdots, 4$,
 and those in $\hat{v}$,  leads to
\begin{equation}\label{3.15}
\begin{split}
c_1+c_3=cb_4,\quad c_1-c_3=cb_2,\\
c_2-c_3=cb_3, \quad c_2+c_3=cb_1.
\end{split}
\end{equation}
 Note that $\hat{v}(\hat{x},1)$ and $\hat{v}(1,\hat{y})$ are equal at corner
$(1,1)$,  $\hat{v}(\hat{x},1)$ and $\hat{v}(-1,\hat{y})$ are equal at corner $(-1,1)$,
 $\hat{v}(\hat{x},-1)$ and $\hat{v}(1, \hat{y})$ are equal at corner $(1,-1)$, $\hat{v}(\hat{x}, -1)$ and  $\hat{v}(-1, \hat{y})$
are equal at corner $(-1,-1)$.
Since $e=L_{2k+1}(1)=-L_{2k+1}(-1)$ and $d=L_{2k+2}(1)=L_{2k+2}(-1)$,  this leads to
\begin{equation}\label{3.16}
\begin{split}
da_1+eb_1=da_4-eb_4 , \quad da_3+eb_3=da_4+eb_4,\\
da_3-eb_3=da_2ecb_2, \quad  da_1-eb_1=da_2-eb_2.
\end{split}
\end{equation}
It follows from \eqref{3.14}, \eqref{3.15} and \eqref{3.16} that
\begin{equation*}
a_1=a_2=a_3=a_4=b_1=b_2=b_3=b_4=c_1=c_2=c_3=0.
\end{equation*}
This implies that $c_4=0$ and $\hat{v}(\hat{x},\hat{y})$ vanishes on the boundary of $\hK$. Hence,
$$
\hat{v}(\hat{x},\hat{y})=\hat{b}(\hat{x},\hat{y})\hat{q}(\hat{x},\hat{y}) \text{ with } \hat{q}(\hat{x},\hat{y})\in P_{2k-3}(\hat{K}),
$$
where the bubble function $\hat{b}(\hat{x}, \hat{y})=(1-\hat{x}^2)(1-\hat{y}^2)$. Finally, since $\hat{v}(\hat{x},\hat{y})$ vanishes at points of $I$, $\hat{q}(\hat{x},\hat{y})\equiv 0$. This  completes the proof.

\end{proof}

\section{The third family of shape function spaces}

For nonconforming elements of even order,  there is always a discrete bubble
   function which vanishes at $4m$ Gauss-Legendre points.
Thus,  in additional to one extra term for $R_m$ of odd $m$,  we need another
   extra term enriching $P_m$ polynomials so that the finite element function
    has exactly $4m$ degrees of freedom on the element boundary.
For $m=2k$,  we define the third family of nonconforming elements by
\an{ \lab{R-m-p}
R_m^{+}(\hat{K}):=P_m(\hat{K})+\text{span}\{\hat{x}^m\hat{y}, \hat{x}\hat{y}^m\}.
}
For the continuity requirement,  we need use the $m=2k$ Gauss-Legendre points,
\a{  g_{-k},\dots, g_{-1}, g_1, \dots, g_{k}. }

\begin{lemma} Let $\hat{v}\in R_m^{+}$ be a polynomial over $\hK$
      such that its restrictions
     on four edges are polynomials of degree $\leq 2k $.
    Then it holds that
\begin{equation}\lab{Rel2}
 \begin{split}
 \sum\limits_{0\ne i=-k}^{k}\frac{ \hat{v}(1,g_i)-\hat{v}(-1,g_i)
    -\hat{v}(g_i,1)+\hat{v}(g_i,-1) }{g_i (1-g_i^2)
       \prod_{|i|\ne j=1}^k (g_i^2-g_j^2) } =0.
\end{split}
 \end{equation}
\end{lemma}

\begin{proof} We add the point $g_{-k-1}=-1$ to the $2k$ Gauss-Legendre points,
and  define a first set of  Lagrange interpolation basis functions:
 \begin{equation*}
 \mathcal{L}_i^0(\hat{x})=\prod\limits_{0,i\not=j=-(k+1)}^{k}
       \frac{\hat{x}-g_j}{g_i-g_j},    \quad
         i=-(k+1), -k, \cdots,-1,1,\cdots,  k.
 \end{equation*}
At the other end, adding a point $g_{k+1}=1$, we
     define another set of  Lagrange interpolation basis functions:
 \begin{equation*}
 \mathcal{L}_i^1(\hat{x})=\prod\limits_{0,i\not=j=-k}^{k+1}
   \frac{\hat{x}-g_j}{g_i-g_j} , \quad i=-k, \cdots,-1,1,\cdots,  k+1.
 \end{equation*}
 Since
 \a{ P_{2k}([-1,1]) &=
    \text{span}\{\mathcal{L}_{i}^0(\hat{x}), \ i=-k-1, \cdots,-1,1,\cdots, k
   \} \\ & =\text{span}\{\mathcal{L}_{i}^1(\hat{x}),
   \ i=-k,  \cdots,-1,1,\cdots,  k+1\} , }
   and  the restriction of  $\hat{v}$  on each edge of $\hK$
    is a polynomial of degree $\leq 2k$,
  we have the four equations in \meqref{2.4}.
Again, let
  \a{ \alpha_i &= \mathcal{L}_i^0(1)=\begin{cases}
          \frac{1-(-1)}{g_i-(-1)} \prod_{0,i\not=j=-k}^{k}
	      \frac{1-g_j}{g_i-g_j}  &\hbox{if } i\ne -k-1, \\
	  \prod_{0\ne j=-k}^{k}
	      \frac{1-g_j}{-1-g_j}  &\hbox{if } i= -k-1, \end{cases}\\
	\beta_i &= \mathcal{L}_i^1(-1)=\begin{cases}
          \frac{-1-1}{g_i-1} \prod_{0,i\not=j=-k}^{k}
	      \frac{-1-g_j}{g_i-g_j}  &\hbox{if } i\ne k+1, \\
	  \prod_{0\ne j=-k}^{k}
	      \frac{-1-g_j}{ 1-g_j}  &\hbox{if } i= k+1. \end{cases}
 }
 Note that, because the Gauss-Legendre points are symmetric,
\a{ \alpha_{-k-1} = \beta_{k+1}, \quad \hbox{and } \
    \alpha_i = -\beta_i \ \hbox{ if } \ i=-k,\cdots, -1,1,\dots, k.
   }
 By the continuity of $\hat v$ at the four corner vertexes,
      it yields
\a{   \alpha_{-k-1}\hat v(-1,-1) +
              \sum_{0\ne i=-k}^k \alpha_{i}\hat v(g_i,-1)
	  & =\beta_{k+1}\hat v( 1, 1) +
              \sum_{0\ne i=-k}^k \beta_{i}\hat v(1,g_i), 
	\\
     \alpha_{-k-1}\hat v(-1,1) +
              \sum_{0\ne i=-k}^k \alpha_{i}\hat v(g_i,1)
	   &=\alpha_{-k-1}\hat v( 1,-1) +
              \sum_{0\ne i=-k}^k \alpha_{i}\hat v(1,g_i), 
	\\
      \beta_{ k+1}\hat v(1,1) +
              \sum_{0\ne i=-k}^k \beta_{i}\hat v(g_i,1)
	   &=\alpha_{-k-1}\hat v(-1,-1) +
              \sum_{0\ne i=-k}^k \alpha_{i}\hat v(-1,g_i), 
	 \\
      \beta_{ k+1}\hat v( 1,-1) +
              \sum_{0\ne i=-k}^k \beta_{i}\hat v(g_i,-1) 
	   &=\beta_{ k+1}\hat v(-1, 1) +
              \sum_{0\ne i=-k}^k \beta_{i}\hat v(-1,g_i). }
Eliminating four corner values of $\hat v$,  we get
  \a{
 \sum\limits_{0\ne i=-k}^{k} \alpha_i \big(\hat{v}(1,g_i)-\hat{v}(-1,g_i)
     -\hat{v}(g_i,1)+\hat{v}(g_i,-1)\big)=0. }
 This is simplified to \meqref{Rel2}.
 \end{proof}

 To define  degrees of freedom for this space, let
 the set of even Gauss-Legendre points on four edges of $\hK$ be
\a{
G^+:=\{(1, g_i),(-1,g_i), (g_i,1), (g_i, -1), \  0\ne i=-k, \cdots, k\}.
  }
 To define other  degrees of freedom in the interior of $\hK$,  let
\a{
  I^+:=\{(\hat{x}_\ell, \hat{y}_\ell), \ell=1, \cdots, (2k-3)(k-1)\}
} be a set of interior points of $\hat{K}$,  the standard Lagrange points
   inside the reference element  $\hat{K}$,
   so that any polynomial $\hat{q}(\hat{x},\hat{y})\in P_{2k-4}(\hat{K})$
    can be uniquely defined by its values at points in $I^+$.

\begin{theorem} If $\hat{v}(\hat{x},\hat{y})\in R_m^+(\hK)$
   vanishes at all points in $G^+ \cup I^+\cup\{(1,1)\}$,
   then $\hat{v}(\hat{x},\hat{y})\equiv 0$.
\end{theorem}

\begin{proof}
The function $\hat{v}(\hat{x},\hat{y})\in R_m^+(\hK)$ is a $P_{2k}$ polynomial
   when restricted to $\hat y=1$.
On the edge $\hat y=1$,  $\hat{v}$ vanishes at $2k$
  Gauss-Legendre points plus a corner point $\{(1,1)\}$.
So  $\hat{v}\equiv 0$ on $\hat y=1$.
Repeating the argument,  as $\hat v$ is continuous at the four corners,
  we find $\hat{v}\equiv 0$ on the whole boundary.
\a{ \hat{v}(\hat{x},\hat{y})   =\hat{b}(\hat{x},\hat{y})\hat{q}(\hat{x},\hat{y})
   \text{ with } \hat{q}(\hat{x},\hat{y})\in P_{2k-4}(\hat{K}),
   }
where  $\hat{b}(\hat{x}, \hat{y})=(1-\hat{x}^2)(1-\hat{y}^2)$.
  Here with a careful division, we find the coefficients for $x^my$ and $y^mx$
   in $\hat v$ are zero.  So $\hat{q}(\hat{x},\hat{y})\in P_{2k-4}(\hat{K}).$
Finally, since $\hat{v}(\hat{x},\hat{y})$ vanishes at points of $I$,
    $\hat{q}(\hat{x},\hat{y})\equiv 0$. This  completes the proof.
\end{proof}

The above theorem implies that any $\hat{v}(\hat{x},\hat{y})\in R_m^+(\hK)$
   can be    uniquely  determined by
    its values at  points in $G^+\cup I^+\cup\{(1,1)\}$.
But the number of points is one more
      than the dimension of $R_m^+(\hat{K})$.
The relation \eqref{Rel2}  implies that the number of linearly
    independent functionals defined
 for the space $R_m^+(\hK)$ is equal to the dimension of the space.
  This  motivates the following degrees of freedom
   for the shape function space $R_m^+(\hK)$:
\begin{itemize}
\item values at points in $G^+$ which satisfy the relation \eqref{Rel2};
\item value  at point  $(1,1)$;
\item values  at points in $I^+$.
\end{itemize}
 Note that the value at point $(1,1)$ is to determine the coefficient of
    the discrete bubble function in $\hat v$:
\a{ \hat b_0 (\hat x, \hat y)  =
    \prod_{i=1}^k ( \hat x^2 +\hat y^2 - 1 - g_i^2). }

\section{Nonconforming finite element spaces}
This section defines  nonconforming quadrilateral element spaces.

\subsection{Quadrilateral Mesh}

Let $\mathcal{T}_h:=\{ K_{i}, i=1,\cdots , Ne\}$ be a shape regular
quadrilateral partition of $\Omega$ with diam$(K_{i})\leq h$.
We assume that the partition $\mathcal{T}_{h}$   satisfies the bisection condition of \cite{Shi84}:
The distance $d_{K}$ between the midpoints of two
diagonals of each element $K$  is of order $\mathcal{O}(h^{2})$.

\vskip 0.2cm
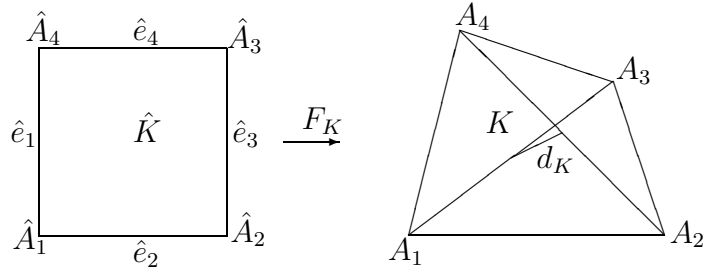
\begin{figure}[h]
\begin{center}
\setlength{\unitlength}{0.25cm}
\begin{picture}(26,13.8)(0,0)
\put(0,1.3){\begin{picture}(22,13)(0,0)
\put (-2,0){\line(0,1){10}}
\put (-2,0){\line(1,0){10}}
\put (-2,10){\line(1,0){10}}
\put (8,0){\line(0,1){10}}
\put(-3.4,-0.7){$\hat{A}_{1}$}
\put(8.2,-0.3){$\hat{A}_{2}$}
\put(8,10){$\hat{A}_{3}$}
\put(-2.7,10.4){$\hat{A}_{4}$}

\put(3,-1.4){$\hat{e}_{2}$}
\put(8.3,5){$\hat{e}_{3}$}
\put(3,10.4){$\hat{e}_{4}$}
\put(-3.5,5){$\hat{e}_{1}$}
\put(3,5){$\hat{K}$}

\put(11,5){\vector(1,0){3}}\put(12,5.7){$F_K$}
\end{picture}
}

\setlength{\unitlength}{0.34cm}

\put(0,-1){\begin{picture}(22,13)(0,0)
\put(13,2){\line(1,4){2}}
\put(13,2){\line(1,0){10}}
\put(23,2){\line(-1,3){2}}
\put(15,10){\line(3,-1){6}}
\put(13,2){\line(4,3){8}}
\put(23,2){\line(-1,1){8}}
\put(17,5){\line(2,1){2}}
\put(12.2,1.0){$A_{1}$}
\put(23.2,1.7){$A_{2}$}
\put(21.1,8.1){$A_{3}$}
\put(14.8,10.2){$A_{4}$}
\put(18,4.6){$d_{K}$}
\put(16,6){$K$}

\end{picture}
}
\end{picture}
\end{center}
\caption{The reference element $\hK$ and a quadrilateral element $K$.}\label{figure1}
\end{figure}

For a given element $K\in \mathcal{T}_h$,  its four  nodes are denoted by
$A_i(x_i,y_i),i=1,\cdots ,4$ in  the counterclockwise order.
Let $\hK:=[-1,1]^2$ denote the reference element with nodes
$\hat{A}_i(\hat{x}_i,\hat{y}_i)$, $i=1,\cdots ,4$, shown in Figure \mref{figure1}.
   Define the bilinear
    transformation  $F_K:\hat{K} \rightarrow K$ by
$$
x=\sum\limits_{i=1}^{4}x_iN_i(\hat{x},\hat{y}),
\quad y=\sum\limits_{i=1}^{4}y_iN_i(\hat{x},\hat{y}),\quad (\hat{x},\hat{y})\in\hat{K},
$$
where $N_i(\hat{x},\hat{y}),i=1,2,3,4$ are the bilinear basis functions,
which can be written as
\ben
N_1(\hat{x},\hat{y})&=&\frac{1}{4}(1-\hat{x})(1-\hat{y}),
\quad N_2(\hat{x},\hat{y})=\frac{1}{4}(1+\hat{x})(1-\hat{y}),\\
N_3(\hat{x},\hat{y})&=&\frac{1}{4}(1+\hat{x})(1+\hat{y}),
\quad N_4(\hat{x},\hat{y})=\frac{1}{4}(1-\hat{x})(1+\hat{y}).
\een

\subsection{Nonconforming finite element spaces and dimensions}

  For an odd integer $m=2k+1>0$, define  nonconforming
 finite element spaces by
 \begin{equation}
 \begin{split} \lab{R-h}
 R_h:=\{v\in L^2(\Omega), v|_{K}\circ F_{K}^{-1}\in R_m(\hK)\text{ for any } K\in\mathcal{T}_h, v\text{ is  continuous }  \\\text{
  at $m$ Gauss-Legendre points of each interior edge of } \mathcal{T}_h \}.
\end{split}
 \end{equation}

\begin{equation}
 \begin{split}\lab{E-R-h-P}
 ER_h^P:=\{v\in L^2(\Omega), v|_{K}\circ F_{K}^{-1}\in ER_m(\hK)\text{ for any } K\in\mathcal{T}_h, v\text{ is  continuous
   } \\ \text{ at $m$ Gauss-Legendre points of each interior edge of } \mathcal{T}_h \}.
\end{split}
 \end{equation}

\begin{equation}
 \begin{split}\lab{E-R-h-M}
 ER_h^M:=\{v\in L^2(\Omega), v|_{K}\circ F_{K}^{-1}\in ER_m(\hK)\text{ for any } K\in\mathcal{T}_h,
  \int_e[v]qds=0\\ \text{ for any }q\in P_{2k}(e) \text{ for each interior edge $e$ of } \mathcal{T}_h \},
\end{split}
 \end{equation}
 where $[v]$ denotes the jump of $v$ across edge $e$.
 We note again that $ER_h^M$ is different from $ER_h^P$ due to
  the higher order polynomial term $(x^{2k+2}- y^{2k+2})$.
For an even integer $m=2k$,  the nonconforming finite element space is defined by
\begin{equation}
 \begin{split} \lab{R-h-p}
 R_h^+:=\{v\in L^2(\Omega), v|_{K}\circ F_{K}^{-1}\in R_m^+(\hK)
      \text{ for any } K\in\mathcal{T}_h, v\text{ is  continuous }  \\
   \text{
  at $m$ Gauss-Legendre points of each interior edge of } \mathcal{T}_h \}.
\end{split}
 \end{equation}

The corresponding homogeneous spaces are defined, respectively,

\an{\lab{R-h-0}
 R_{h,0}:&=\{v \in R_h, v\text{ vanishes  at $m$ Gauss-Legendre points }\\
    \nonumber&\qquad
    \text{ of each boundary edge $e$ of } \mathcal{T}_h \}.
     \\
      \lab{ER-P-h-0}
 ER_{h, 0}^P&:=\{v\in ER_h^P, v\text{ vanishes  at $m$ Gauss-Legendre points }\\
    \nonumber&\qquad
    \text{ of each boundary edge $e$ of } \mathcal{T}_h\}.
    \\
  \lab{ER-M-h-0}
 ER_{h, 0}^M&:=\{v\in ER_h^M, \int_e v qds=0 \text{ for any }q\in P_{2k}(e) \\
   \nonumber&\qquad
   \text{ for each boundary edge $e$ of } \mathcal{T}_h \}.
   \\
     \lab{R-h-p-0}
 R_{h,0}^+&:=\{v \in R_h^+, v\text{ vanishes  at $m$ Gauss-Legendre points }\\
    \nonumber&\qquad
    \text{ of each boundary edge $e$ of } \mathcal{T}_h \}.
   }

\subsection{Approximations of nonconforming finite element spaces}

Given $K\in\mathcal{T}_h$,  define
\begin{equation}
G_K:=G\circ F_{K}, \text{ and }I_K:=I\circ F_K.
\end{equation}
Then,  define  the canonical interpolation operator $\Pi_{ER}^P:
    H^2(\Omega)\rightarrow ER_h^P$ by
\an{ \lab{I-P}
   (\Pi_{ER}^P v|_K)(p)=v|_K(p) \text{ for any }p\in G_K\cup I_K
     \text{ and }K\in\mathcal{T}_h
     }
for any $v\in H^2(\Omega)$.

Define  the canonical interpolation operator
    $\Pi_{ER}^M: H^2(\Omega)\rightarrow ER_h^P$ by
\an{\lab{I-M}
  \ad{ (\Pi_{ER}^M  v|_K)(p) &=v|_K(p) \quad \text{ for any }p\in I_K,\\
  \\\int_e \Pi_{ER}^M  v|_e ds&=\int_e vds \quad
     \text{ for any }e\subset \partial K,
    } }
 for any $K\in\mathcal{T}_h$ and  $v\in H^2(\Omega)$.

To define an interpolation operator for $R_h$,  let $\Pi^Q$ be
     the canonical interpolation operator of the
conforming $Q_m$ element space $Q_{m,h}:=\{v\in H^1(\Omega), v|_K\circ F_K^{-1} \in Q_m(\hK), K\in \mathcal{T}_h \}$.
 Then,  define  an interpolation operator $\Pi_R: H^2(\Omega)\rightarrow R_h$ by
\an{\lab{I-R}
(\Pi_R v|_K)(p)=\Pi^Q v|_K(p) \text{ for any }p\in G_K \cup I_K
     \text{ and }K\in\mathcal{T}_h
     }
for any $v\in H^2(\Omega)$. Since the values at points in $G_K$ of $\Pi^Q v$
   satisfy the relation \eqref{Rel}, this operator
     is well-defined.

For the even order nonconforming finite elements,
   we define
\an{\lab{I-R-P}  (\Pi_{R^+} v|_K)(p)=\Pi^Q v|_K(p) \quad
    \text{ for any }p\in G^+_K \cup I^+_K\cup \{F_K(1,1)\}
     \text{ and } K\in\mathcal{T}_h
     }
for any $v\in H^2(\Omega)$. Since the values of
   $Q_m$ polynomial  $\Pi^Q v$ at points in $G_K^+$ satisfy
   the constraint \meqref{Rel2}, the
   operator $\Pi_{R^+} :  H^2(\Omega) \to  R_h^+$ is well-defined.

An immediate  consequence of these interpolation operators is the
    following approximation property
\begin{equation}
\inf\limits_{v_h\in V_h}\|\nabla_h (u-v_h)\|_0\leq Ch^m\|u\|_{m+1},
\end{equation}
provided that $u\in H^{m+1}(\Omega)$ and the mesh satisfies the
  bisection condition where $V_h=R_h, ER_h^P, ER_h^M, R_h^+$.

\subsection{The full space of nonconforming finite elements}

For the finite element spaces $ER_{h,0}^P$ and $ER_{h,0}^M$,  local  nodal basis
  functions are uniquely determined,  which are glued together to form
  the global nodal basis functions.
Thus,  the interpolation operators $\Pi_{ER}^P$ and $\Pi_{ER}^M$,
   cf. \meqref{I-P} and  \meqref{I-M}, are on-to mappings, by which
    the dimension of the finite element spaces can be counted.

Let $N_{V}$, $N_S$, and $N_{E}$  denote the numbers of vertexes, edges
    and elements of the partition ${\mathcal T}_h$, respectively.
Let $N_{V}^i$, $N_V^b$, $N_S^i$  and  $N_{S}^{b}$
   denote the numbers of interior vertexes, interior edges,
    boundary vertexes  and boundary edges,
    respectively.
 The dimensions are
\begin{equation}
     \dim ER_{h,0}^P =\dim ER_{h,0}^M =
      N_E (2k-1)(k-1)+ N_S^i(2k+1).
\end{equation}

However,  for nonconforming spaces $R_{h,0}$ and $R_{h,0}^+$,
   the local nodal basis is not unique.
The canonical basis functions are not linearly independent but subject
   to the constraints \meqref{Rel} and \meqref{Rel2}.
When gluing these functions together to form the global space,
   it is not clear if they can form a basis for the full space.
In other words,  the local constraints  \meqref{Rel} (on each element)
   might be linearly dependent globally, that is,  some constraints
   may be automatically satisfied when the neighboring elements are subject
  to the constraints.
If so,  the interpolation operators $\Pi_{R} $ and $\Pi_{R^+} $
   in \meqref{I-R} and \meqref{I-R-P} are not on-to mappings.
Then,  the computational finite element spaces are not the full spaces
   defined mathematically in \meqref{R-h} and \meqref{R-h-p}, but subspaces.
This problem is not well-addressed in previous research.
We give a rigorous analysis.

\begin{theorem} The interpolation operators are on-to mappings that
  \an{ \lab{full} \Pi_{R}(Q_{m,h} \cap H^1_0(\Omega)) &= R_{h,0}, \\
       \Pi_{R^+}(Q_{m, h}\cap H^1_0(\Omega)) &= R^+_{h,0}.
   }
   The dimensions of these spaces are
       \an{\lab{d-0} \dim  R_{h,0} &=
            N_E ((2k-1)(k-1) ) + N_V^i + N_S^i (2k) , \\
	    \dim  R^+_{h,0} &=
            N_E ((2k-3)(k-1)+1) + N_V^i + N_S^i  (2k-1) .
	    }
\end{theorem}

\begin{proof}  We analyze the space $R_{h,0}$ only as that for $R^+_{h,0}$
  is the same.   $R_{h,0}$ is defined as the full
   space.  But the range of the interpolation operator is a subspace
   \a{  \Pi_R (Q_{m, h} \cap H^1_0(\Omega)) \subset
     R_{h,0}. }
   To show that the two spaces are equal,  we show that
    the dimension of the range of $\Pi_R$ is no less than
    that of the full space $R_{h,0}$.
   The former is easily counted by the definition of the operator
      $\Pi_R$ in \meqref{I-R},
   \an{\lab{d-r}\dim \text{Range}(\Pi_R)=
            N_E ((2k-1)(k-1) ) + N_V^i + N_S^i (2k). }

   We define a completely discontinuous space
   \a{  R_{h}^d = \{ v\in L^2(\Omega) \mid
       v|_K\circ F_{K}^{-1} \in R_m(\hat K) \quad \forall K\in\mathcal{T}_h \}. }
    Then\an{\lab{d-d}
        \dim R_{h}^d = N_E \cdot (\dim P_{2k+1} + 1 )
                  = N_E ( (2k+3)(k+1)+1 ). }
    Now, we introduce linear functionals on the space $R_h^d$:
   \an{   f_i \ &: \ R_h^d \to R^1, \\
       \lab{f-i}  f_i(v) &= \begin{cases}
	   v(g_i)
	     & \text{ for any } g_i \in G_K\cap \partial\Omega,
	      \\ v(g_i^+)-v(g_i^-)
	     & \text{ for rest } g_i \in G_K\cap  \Omega^i, \end{cases}
	 }
for any $K \in  \mathcal{T}_h$,
   where we randomly choose a plus side for each edge, and
    $v(g_i^+)$ and $v(g_i^+)$ are values of $v$ on the two sides
     of an edge where $g_i$ is a Gauss-Legendre point, defined
    in \meqref{G}.
  The nonconforming finite element space $R_h$ is the kernel space
   of the space of above linear functionals.

  Assume the domain is simply connected.  Let $g_i$ be any one
   boundary Gauss-Legendre point.  Let the linear functional
   evaluating at this point be $f_0$.   We show next that
   value of $f_0(v)=v(g_i)$ is completely determined by the
   rest nodal values for all $v\in R_h$.    For example,
    if the domain consists of one element $K$,
  then $v(g_i)=\gamma_i^{-1}\sum_{j\ne i}-\gamma_j v(g_j)$,
    cf. \meqref{Rel}.
   We define a set for the rest linear functionals in \meqref{f-i}
    \an{\lab{set} \mathcal{F}=\{ f_i \mid  i=1,2,\cdots, N_S(2k+1)-1 \}.}
Assume
    \an{\lab{dependent}  \sum_{i=1}^{\dim \mathcal{F} } c_i f_i =0. }
    That is
	\a{     \sum_{i=1}^{\dim \mathcal{F}} c_i f_i(v_h) = 0
	\quad \forall v_h \in  R_{h}^d. }
  For convenience,  we let $c_0=0$.
   Let $v_h\in R_h^d$ and $v_h=0$ on all elements except the element where
     $f_0$ is defined.
   We have then
   \a{   0=\sum_{i=0}^{\dim \mathcal{F}} c_i f_i(v_h) = \sum_{i=0}^{4m-1} c_i (\pm v_h(g_i)).}
 By the constraint \meqref{Rel},
   \a{ c_i = c\gamma_i  \hbox{ \ or \ } -c\gamma_i }
 for some uniform constant $c$.  But above $\{c_i\}$ contains $c_0$ which is $0$.
 Thus all $c_i$ related to the element $K$ are $0$.
 In this fashion, we know that all $c_i$ in \meqref{dependent} are $0$ and
   the set $\mathcal{F}$ is linearly independent.
  Then the dimension of its kernel is $\dim \mathcal{F}$.
  By \meqref{d-d} and \meqref{set}, \meqref{d-0} holds,
  \a{ \dim R_{h,0} &= \dim R_h^d - \dim \mathcal{F}\\
      &=N_E ( (2k+3)(k+1)+1 ) - (N_S(2k+1)-1) \\
      &= N_E ( (2k-1)(k-1) )+ N_S^i(2k) + N_V^i, }
where we used the relations $4N_E=2N_S^i + N_S^b$,  $2 N_E = 2N_V^i +　N_V^b -2$
   and $ N_V^b = N_S^b $. \meqref{full} follows \meqref{d-r} and \meqref{d-0}.
\end{proof}

\section{The analysis of consistency errors}

 To analyze  consistency errors, we need some additional
   interpolation operators.
On the interval $[-1,1]$, define two $L^2$ projection operators
    $\mathcal{P}_{m-i}: L^2([-1,1])\rightarrow P_{m-i}([-1,1])$,
for $i=1, 2$, respectively,
\begin{equation}\lab{p-m-i}
\int_{-1}^1\mathcal{P}_{m-i} w\, qds=\int_{-1}^1 w\, qds
    \text{ for any }q\in P_{m-i}([-1,1]), i=1,2.
\end{equation}
for any $w\in L^2([-1, 1])$. Define the interpolation
     $\Pi_{m-1}: C^0([-1, 1])\rightarrow P_{m-1}([-1, 1])$ by
\begin{equation}\lab{pi-m-1}
(\Pi_{m-1}v)(g_i)=v(g_i), i=-k, \cdots, k, \text{ for any }v\in C^0([-1, 1]),
\end{equation}
where $g_i$ are Gauss-Legendre points on the interval $[-1, 1]$.
We recall that $\hat{e}_1$ and $\hat{e}_3$ be two edge of $\hK$
    that parallel to the $\hat{y}$  axis, see Figure \ref{figure1}.

Since the analysis of  consistency errors for both $R_{h, 0}$ and $ER_{h, 0}^P$ is similar, only details for $ER_{h,0}^P$
 are presented as follows.

\begin{lemma}\label{Lemma4.1} Let $\Pi_{m-1}$ be the interpolation operator defined
  in \meqref{pi-m-1}.
 Then it holds that
\begin{equation*}
\bigg|\int_{\hat{e}_3} \hat{u} \, \big(\hat{v} -\Pi_{m-1}\hat{v}|_{\hat{e}_3}\big)d\hat{y}
 -\int_{\hat{e}_1} \hat{u} \, \big(\hat{v} -\Pi_{m-1}\hat{v}|_{\hat{e}_1}\big) d\hat{y}\bigg|
\leq C|\hat{u}|_{H^{m}(\hK)}|\hat{v}|_{H^{m}(\hK)},
\end{equation*}
for any $\hat{u}\in H^m(\hK)$ and $\hat{v}\in ER_m(\hK)$.
\end{lemma}
\begin{proof} Let $\mathcal{P}_{m-1}$ be the $L^2$ projection operator defined
   in \meqref{p-m-i}. By which, we use the following decomposition
\a{
\int_{\hat{e}_3} \hat{u} \, \big(\hat{v} -\Pi_{m-1}\hat{v}|_{\hat{e}_3}\big)d\hat{y}
 -\int_{\hat{e}_1} \hat{u} \, \big(\hat{v} -\Pi_{m-1}\hat{v}|_{\hat{e}_1}\big) d\hat{y}
 =:I_1+I_2}
where
\a{I_1 &=\int_{-1}^1(I-\mathcal{P}_{m-1})\hat{u}|_{\hat{e}_3}\, \big(\hat{v}|_{\hat{e}_3}
    -\Pi_{m-1}\hat{v}|_{\hat{e}_3}\big)d\hat{y}\\
   &\qquad
-\int_{-1}^1(I-\mathcal{P}_{m-1})\hat{u}|_{\hat{e}_1}\,
   \big(\hat{v}|_{\hat{e}_1}-\Pi_{m-1}\hat{v}|_{\hat{e}_1}\big) d\hat{y}, \\
  I_2&= \int_{-1}^1\mathcal{P}_{m-1}\hat{u}|_{\hat{e}_3}\,
       \big(\hat{v}|_{\hat{e}_3}-\Pi_{m-1}\hat{v}|_{\hat{e}_3}\big)d\hat{y}\\
   &\qquad -\int_{-1}^1\mathcal{P}_{m-1}\hat{u}|_{\hat{e}_1}\,
     \big(\hat{v}|_{\hat{e}_1}-\Pi_{m-1}\hat{v}|_{\hat{e}_1}\big) d\hat{y} .
    }
The first term $I_1$ can be bounded by the Cauchy-Schwarz inequality,
the trace theorem and the usual Bramble-Hilbert lemma,
 \begin{equation*}
   |I_1|\leq C|\hat{u}|_{H^{m}(\hK)}|\hat{v}|_{H^m(\hK)}.
 \end{equation*}
 To analyze the second term $I_2$, introduce the following decomposition for $\hat{v}$
 \begin{equation*}
 \hat{v}=\hat{v}_1+\hat{v}_2,
 \end{equation*}
 where $\hat{v}_1\in P_{m}(\hK)+\text{span}\{\hat{x}^{2k+1}\hat{y}
    -\hat{x}\hat{y}^{2k+1}\}$, and $\hat{v}_2=c_0(\hat{x}^{2k+2}-\hat{y}^{2k+2})$
   with $c_0$,  an interpolation constant.
Since $\hat{v}_1|_{\hat{e}_i}-\Pi_{m-1}\hat{v}_1|_{\hat{e}_i}$, $i=1, 3$,
   are  polynomials of degree at most $2k+1$,
 and vanish at $2k+1$ Gauss-Legendre points of $\hat{e}_i$,
  \begin{equation*}
  \hat{v}_1|_{\hat{e}_i}-\Pi_{m-1}\hat{v}_1|_{\hat{e}_i}=c_iL_{2k+1}(\hat{y}),
  \end{equation*}
     with two constants $c_i$, $i=1, 3$. Since  degrees of polynomials
     $\mathcal{P}_{m-1}\hat{u}|_{\hat{e}_i}$, $i=1, 3$, are
     not more than $2k$,  this gives
  \begin{equation}\label{3.10}
  \int_{-1}^1\mathcal{P}_{m-1}\hat{u}|_{\hat{e}_3}\,
     \big(I-\Pi_{m-1}\big)\hat{v}_1|_{\hat{e}_3}d\hat{y}
   -\int_{-1}^1\mathcal{P}_{m-1}\hat{u}|_{\hat{e}_1}\,
    \big(I-\Pi_{m-1}\big)\hat{v}_1|_{\hat{e}_1} d\hat{y}=0.
  \end{equation}
  Since $\hat{v}_2|_{\hat{e}_3}=\hat{v}_2|_{\hat{e}_1}$,
     it holds that $(I-\Pi_{m-1})\hat{v}_2|_{\hat{e}_3}
     =(I-\Pi_{m-1})\hat{v}_2|_{\hat{e}_1}$.
 Hence,
  \begin{equation}\label{3.11}
  \begin{split}
&\quad \ \int_{-1}^1\mathcal{P}_{m-1}\hat{u}|_{\hat{e}_3}\,
      \big(I-\Pi_{m-1}\big)\hat{v}_2|_{\hat{e}_3}d\hat{y}
   -\int_{-1}^1\mathcal{P}_{m-1}\hat{u}|_{\hat{e}_1}\,
    \big(I-\Pi_{m-1}\big)\hat{v}_2|_{\hat{e}_1} d\hat{y}\\
&=\int_{-1}^1\mathcal{P}_{m-1}\big(\hat{u}|_{\hat{e}_3}-\hat{u}|_{\hat{e}_1}\big)\,
  \big(I-\Pi_{m-1}\big)\hat{v}_2|_{\hat{e}_3}d\hat{y}.
  \end{split}
  \end{equation}
From  facts that  $(I-\Pi_{m-1})\hat{v}_2|_{\hat{e}_3}$
     is a polynomial of degree $\leq 2k+2$ and vanishes
at $2k+1$ Gauss-Legendre points of $\hat{e}_3$, it follows
\begin{equation}\label{3.12}
\begin{split}
&\quad \
    \int_{-1}^1\mathcal{P}_{m-1}\big(\hat{u}|_{\hat{e}_3}-\hat{u}|_{\hat{e}_1}\big)\,
    \big(I-\Pi_{m-1}\big)\hat{v}_2|_{\hat{e}_3}d\hat{y}\\
&=\int_{-1}^1(I-\mathcal{P}_{m-2})\mathcal{P}_{m-1}
    \big(\hat{u}|_{\hat{e}_3}-\hat{u}|_{\hat{e}_1}\big)\,
       \big(I-\Pi_{m-1}\big)\hat{v}_2|_{\hat{e}_3}d\hat{y}.
\end{split}
\end{equation}
It is straightforward to see that the right hand side of
     \eqref{3.12} vanishes for all $\hat{u}\in P_{m-1}(\hK)$, which
leads to
\begin{equation}\label{3.13}
\begin{split}
&\quad \ \bigg| \int_{-1}^1\mathcal{P}_{m-1}\big(\hat{u}|_{\hat{e}_3}
       -\hat{u}|_{\hat{e}_1}\big)\, \big(I-\Pi_{m-1}\big)\hat{v}_2|_{\hat{e}_3}d\hat{y}\bigg|\\
&\leq C\,|c_0||\hat{u}|_{H^m(\hK)}\leq C |\hat{u}|_{H^m(\hK)} |\hat{v}|_{H^{m+1}(\hK)}.
\end{split}
\end{equation}
A summary of \eqref{3.10}, \eqref{3.11}, \eqref{3.12}, and \eqref{3.13} proves
\begin{equation*}
|I_2|\leq C |\hat{u}|_{H^m(\hK)} |\hat{v}|_{H^{m}(\hK)},
\end{equation*}
which completes the proof.
\end{proof}

Let $\hat{e}_2$ and $\hat{e}_4$ be two edge of $\hK$ that parallel to the $\hat{x}$  axis, see figure \ref{figure1}.
 A similar argument of the above lemma can prove the following result.

\begin{lemma}\label{Lemma4.2}
Let $\Pi_{m-1}$ be the interpolation operator defined in \meqref{pi-m-1}.
   Then it holds that
  \a{
\bigg |\int_{\hat{e}_4}\hat{u} \big(\hat{v} -\Pi_{m-1}\hat{v}|_{\hat{e}_4}\big)d\hat{x}
  -\int_{\hat{e}_2} \hat{u} \big(\hat{v} -\Pi_{m-1}\hat{v}|_{\hat{e}_2}\big) d\hat{x}\bigg|
\leq C|\hat{u}|_{H^{m}(\hK)}|\hat{v}|_{H^{m}(\hK)}. }
for any $\hat{u}\in H^m(\hK)$ and $\hat{v}\in ER_m(\hK)$.
\end{lemma}

\begin{theorem} Suppose that  the quadrilateral mesh $\cT_h$ satisfies the bi-section condition. Then it holds that
\begin{equation}
\sup\limits_{0\not=v_h\in ER_{h,0}^P}\frac{\sum\limits_{e\in\mathcal{E}_{h}}\int_e \frac{\partial u}{\partial n}[v_h]ds}{\|\nabla_h v_h\|_0}
\leq Ch^{m}|u|_{m+1},
\end{equation}
for any $u\in H^{m+1}(\Omega)$.
\end{theorem}
\begin{proof} Based on Lemmas \ref{Lemma4.1} and \ref{Lemma4.2}, the proof  follows from similar procedures used in \cite{HuShi(2005),Shi84}.
\end{proof}
\begin{remark}
Similar arguments  can show  similar estimates for consistency
   error for spaces $R_{h,0}$, $ER_{h,0}^M$ and $R^+_{h,0}$.
\end{remark}

\begin{theorem}  Let $u$ and $u_h$ be the exact solution and finite element solution
   of the Poisson equation,  respectively,
 \a{  (\nabla u, \nabla v) &= (f, v) \quad \forall v \in H^{r+1}(\Omega)
             \cap H^1_0(\Omega), \\
	 ( \nabla u_h, \nabla v) &= (f, v) \quad \forall v \in V_h, }
   where $V_h= R_{h,0}$, $ER_{h,0}^P$, $ER_{h,0}^M$ or $R^+_{h,0}$.
  Then \a{  |u-u_h|_{H^1_h} \le C h^{\min \{r, m\} }  |u|_{H^{r}(\Omega)},}
  where $m=2k-1$ for $V_h= R_{h,0}$, $ER_{h,0}^P$ or $ER_{h,0}^M$,
       and $m=2k$ for $V_h=R^+_{h,0}$.
\end{theorem}

\begin{proof}  It is standard, by applying the Strang lemma,
    cf. \cite{Ciarlet(1978)}.
\end{proof}
	
\section{Numerical test}

We compute the $P_k$ nonconforming finite element
 solutions on uniform rectangular grids for the
 following Poisson equation:
 \a{ - \Delta u &= f \quad\hbox{ in } \ \Omega=(0,1)^2,\\
         u &=0\quad\hbox{ on } \ \partial\Omega.}
In computation, the exact solution is,
         \an{\lab{s} u(x,y)=2^4 (x-x^6)(y-y^2). }
The first level grid is the unit square.
Each subsequent grid is a refinement of the last one by
  dividing each square into 4, denoted by $\{{\mathcal T}_h\}$.
We use the following 6 nonconforming finite elements:
\an{\lab{V-3E} { \ } &&
 V^{(3E)}_h&=\{v\in L^2 \ \mid \
        v|_K\in P_3\oplus\{xy^3-x^3y, x^{4}-y^4\}, \
        [v]_e\perp_P P_2 \},  &{ \ } &{ \ }  \\
 \lab{V-4}&&
 V^{(4)}_h&=\{v\in L^2 \ \mid \
        v|_K\in P_4\oplus\{xy^4,x^4y\}, \
        [v]_e\perp P_3 \},\\
 \lab{V-5}&&
 V^{(5)}_h&=\{v\in L^2 \ \mid \
        v|_K\in P_5\oplus\{xy^5\}, \
        [v]_e\perp P_4 \},\\
  \lab{V-5E}&&
  V^{(5E)}_h&=\{v\in L^2 \ \mid \
        v|_K\in P_5\oplus\{xy^5-x^5y, x^6-y^6\}, \
        [v]_e\perp_P P_4\},\\
 \lab{V-6}&&
 V^{(6)}_h&=\{v\in L^2 \ \mid \
        v|_K\in P_6\oplus\{xy^6,x^6y\}, \
        [v]_e\perp P_5 \},\\
 \lab{V-7}&&
 V^{(7)}_h&=\{v\in L^2 \ \mid \
        v|_K\in P_7\oplus\{xy^7\}, \
        [v]_e\perp P_6 \},
 }
   where $K$ is any square element in ${\mathcal T}_h$ and
    $e$ is any edge in the grid ${\mathcal T}_h$.
To be precise,
  \a{  V^{(3E)}_h&= ER^P_{h,0}, &&\text{defined in \meqref{ER-P-h-0}
           with $m=3$, }\\
        V^{(4)}_h&=  R^+_{h,0}, &&\text{defined in \meqref{R-h-p-0}
           with $m=4$, }\\
        V^{(5)}_h&= R_{h,0}, &&\text{defined in \meqref{R-h-0}
           with $m=5$, }\\
        V^{(5E)}_h&= ER^P_{h,0}, &&\text{defined in \meqref{ER-P-h-0}
           with $m=5$, }\\
        V^{(6)}_h&=  R^+_{h,0}, &&\text{defined in \meqref{R-h-p-0}
           with $m=6$, }\\
         V^{(7)}_h&= R_{h,0}, &&\text{defined in \meqref{R-h-0}
           with $m=7$. }
	    }
 We list the computation results in Tables \mref{b-3}--\mref{b-7}.
In all cases except the case of polynomial degree 3,	
  we stop the computation when the machine accuracy is
  reached, i.e., the relative error of computed solutions is
   about $10^{-15}$.
We can see,  for all the finite element spaces,   the
computational solutions converge at the optimal
  order of rate.  This is proved in paper.

 \begin{table}[htb]
  \caption{ \lab{b-3} The error and order of convergence
     by the 2D $C^{-1}$-$P_3$ element \meqref{V-3E} for \meqref{s}.}
\begin{center}  \begin{tabular}{c|rr|rr}  
\hline &  $ \| u - u_h\|_{L^2}$&$h^r$ &
   $| u - u_h|_{H^1_h}$&$h^r$   \\ \hline
 1&  0.000000000&0.0&  0.00000000&0.0\\
 2&  0.172089821&0.0&  1.15734862&0.0\\
 3&  0.012510804&3.8&  0.16038663&2.9\\
 4&  0.000823397&3.9&  0.02155950&2.9\\
 5&  0.000052434&4.0&  0.00280349&2.9\\
 6&  0.000003300&4.0&  0.00035752&3.0\\
 7&  0.000000207&4.0&  0.00004514&3.0\\
 8&  0.000000013&4.0&  0.00000567&3.0\\
   \hline
\end{tabular}\end{center} \end{table}

 \begin{table}[htb]
  \caption{ \lab{b-4} The error and order of convergence
     by the 2D $C^{-1}$-$P_4$ element \meqref{V-4} for \meqref{s}.}
\begin{center}  \begin{tabular}{c|rr|rr}  
\hline &  $ \| u - u_h\|_{L^2}$&$h^r$ &
   $| u - u_h|_{H^1_h}$&$h^r$   \\ \hline
 1&  1.520882827023&0.0& 10.1100205010&0.0\\
 2&  0.073186215065&4.4&  0.9404045783&3.4\\
 3&  0.002467158503&4.9&  0.0655417961&3.8\\
 4&  0.000078057209&5.0&  0.0042062381&4.0\\
 5&  0.000002441049&5.0&  0.0002645249&4.0\\
 6&  0.000000076219&5.0&  0.0000165552&4.0\\
 7&  0.000000002381&5.0&  0.0000010349&4.0\\
   \hline
\end{tabular}\end{center} \end{table}

 \begin{table}[htb]
  \caption{ \lab{b-5} The error and order of convergence
     by the 2D $C^{-1}$-$P_5$ element \meqref{V-5} for \meqref{s}.}
\begin{center}  \begin{tabular}{c|rr|rr}  
\hline &  $ \| u - u_h\|_{L^2}$&$h^r$ &
   $| u - u_h|_{H^1_h}$&$h^r$   \\ \hline
 1&  0.162340154&0.0&  0.94025897&0.0 \\
 2&  0.003652811&5.5&  0.05082935&4.2 \\
 3&  0.000061390&5.9&  0.00176166&4.9 \\
 4&  0.000000983&6.0&  0.00005729&4.9 \\
 5&  0.000000016&6.0&  0.00000182&5.0 \\
 6&  0.000000000&6.0&  0.00000006&5.0 \\
   \hline
\end{tabular}\end{center} \end{table}

 \begin{table}[htb]
  \caption{ \lab{b-5E} The error and order of convergence
     by the 2D $C^{-1}$-$P_5$ element \meqref{V-5E} for \meqref{s}.}
\begin{center}  \begin{tabular}{c|rr|rr}  
\hline &  $ \| u - u_h\|_{L^2}$&$h^r$ &
   $| u - u_h|_{H^1_h}$&$h^r$   \\ \hline
 1&  0.160588210&0.0&  0.93096655&0.0\\
 2&  0.003712721&5.4&  0.05204506&4.2\\
 3&  0.000062480&5.9&  0.00180745&4.8\\
 4&  0.000000999&6.0&  0.00005869&4.9\\
 5&  0.000000016&6.0&  0.00000186&5.0\\
 6&  0.000000000&6.0&  0.00000006&5.0\\
   \hline
\end{tabular}\end{center} \end{table}

 \begin{table}[htb]
  \caption{ \lab{b-6} The error and order of convergence
     by the 2D $C^{-1}$-$P_6$ element \meqref{V-6} for \meqref{s}.}
\begin{center}  \begin{tabular}{c|rr|rr}  
\hline &  $ \| u - u_h\|_{L^2}$&$h^r$ &
   $| u - u_h|_{H^1_h}$&$h^r$   \\ \hline
1&  0.051014969254&0.0&  0.4290029114&0.0 \\
2&  0.000428314992&6.9&  0.0080196383&5.7 \\
3&  0.000003402349&7.0&  0.0001296547&6.0 \\
4&  0.000000026614&7.0&  0.0000020505&6.0 \\
5&  0.000000000209&7.0&  0.0000000322&6.0 \\
6&  0.000000000021&3.3&  0.0000000012&4.8 \\
  \hline
\end{tabular}\end{center} \end{table}

 \begin{table}[htb]
  \caption{ \lab{b-7} The error and order of convergence
     by the 2D $C^{-1}$-$P_7$ element \meqref{V-7} for \meqref{s}.}
\begin{center}  \begin{tabular}{c|rr|rr}  
\hline &  $ \| u - u_h\|_{L^2}$&$h^r$ &
   $| u - u_h|_{H^1_h}$&$h^r$   \\ \hline
 1&  0.012193263916&0.0&  0.1177979140&0.0 \\
 2&  0.000046859707&8.0&  0.0009292948&7.0 \\
 3&  0.000000182695&8.0&  0.0000072707&7.0 \\
 4&  0.000000000839&7.8&  0.0000000571&7.0 \\
   \hline
\end{tabular}\end{center} \end{table}

\end{document}